\documentclass[11pt]{amsart}
\usepackage{a4wide}
\usepackage{amssymb,amsthm}
\usepackage{fullpage}
\usepackage{color}
\usepackage{amsmath}

\newcommand{\beq}{\begin{equation}}
\newcommand{\eeq}{\end{equation}}

\newcommand{\R}{{\mathbb R}}         
\newcommand{\Z}{{\mathbb Z}}

\def\eps{\varepsilon}

\Large

\newtheorem{prop}{Proposition}
\newtheorem{theoreme}[prop]{Theorem}
\newtheorem{lem}[prop]{Lemma}

\newtheorem{rem}[prop]{Remark}

\numberwithin{equation}{section}
\numberwithin{prop}{section}
\begin{document}
\title{ Stability and instability  of the KdV solitary wave  under the KP-I flow}
\author{Frederic Rousset}
\address{IRMAR, Universit\'e de Rennes 1, campus de Beaulieu,  35042  Rennes cedex, France}
\email{ frederic.rousset@univ-rennes1.fr  }
\author{Nikolay Tzvetkov}
\address{University of Cergy-Pontoise, UMR CNRS 8088, Cergy-Pontoise, F-95000 and IUF}
\email{nikolay.tzvetkov@u-cergy.fr}
\date{}
\begin{abstract}

We consider the KP-I  and  gKP-I  equations in $\mathbb{R}\times ( \mathbb{R}/2\pi \mathbb{Z})$. We prove that
    the KdV soliton with subcritical speed $0<c<c^*$  is orbitally
stable under the  global KP-I flow  constructed  by Ionescu and Kenig \cite{IK}.  For supercritical speeds $c>c^*$,  in the spirit of the work by Duyckaerts and Merle
\cite{DM}, 
we  sharpen our previous instability result and
 construct a  global solution which is  different from the solitary wave and its translates  and which converges to the solitary wave as time  goes to infinity.
This last result also holds  for the gKP-I equation.
\end{abstract}
\maketitle
\section{Introduction}
Consider the KP-I equation
\begin{equation}\label{KPI}
\partial_t u+u\partial_x u+\partial_x^3 u-\partial_x^{-1}\partial_y^2u=0,\quad x\in\R,\,\, y\in S^1\equiv \R\slash (2\pi\Z).
\end{equation}
The antiderivative $\partial_{x}^{-1}$  in \eqref{KPI} is defined as the multiplication of the Fourier transform with the (singular) factor $(i\xi)^{-1}$. 
This equation introduced in \cite{K-P} is a model equation  for the propagation of long waves weakly modulated
 in the transverse direction which arises in many physical models. For example, it is an asymptotic model for  the propagation of capillary-gravity waves, 
  (see  \cite{Alvarez-Lannes} for example), or for travelling waves in Bose-Enstein condensates (see \cite{Chiron-Rousset} for example).
   
Let us consider the KdV soliton
\begin{equation}\label{soliton}
S_{c}(t,x)=c\,Q(\sqrt{c}(x-ct)),\quad c>0,\quad Q(x)=3\,{\rm ch}^{-2}(x/2)\,.
\end{equation}
We have that $S_c$ is a solution of the KP-I equation \eqref{KPI}. It is independent of $y$ and thus it is also a solution of
the KdV equation
$$
\partial_t u+u\partial_x u+\partial_x^3 u=0.
$$
The following  classical result is due to Benjamin \cite{Be}.
\begin{theoreme}[\cite{Be}]
$S_c(t,x)$ is orbitally stable as a solution of the KdV equation. 
\end{theoreme}
It is then a natural and physically relevant question to ask whether $S_c(t,x)$ is stable as a solution of the KP-I equation \eqref{KPI}.
As will show here the answer is positive if $c$ is small enough while it is known that it is unstable for large $c$ (see below and \cite{RT1,Z2}).

Before addressing the stability issues one should construct a global flow of \eqref{KPI} in a suitable functional framework. 
For that  purpose, following \cite{IK}, we
define the spaces $Z^s=Z^s(\R\times S^1)$ as 
$$
Z^s=\{u\,:\, \|(1+|\xi|^s+|\xi^{-1}k|^s)\hat{u}(\xi,k)\|_{L^2(\R_\xi\times \Z_k)}<\infty\}
$$
 equipped with the natural norm. The following result is due to Ionescu and Kenig.
\begin{theoreme}[\cite{IK}]
\label{IK}
The KP-I equation \eqref{KPI} is globally well-posed for data in $Z^2(\R\times S^1)$.
\end{theoreme}
The spaces $Z^s$ are closely related to the structure of the KP-I conservation laws (see \cite{ZS}). 
The above global well-posedness requires  the use of the  three first invariants, nevertheless, in this paper,  we shall 
 only  need to manipulate the two first  KP-I conservation laws. 
Let us recall them. By multiplying \eqref{KPI} by $u$ and integrating over $\R\times S^1$, we obtain that the $L^2$ norm is formally conserved under 
the KP-I flow. Next, by multiplying \eqref{KPI} by $\frac{1}{2}u^2+\partial_x^2u-\partial_{x}^{-2}\partial_{y}^2u$ and an integration over $\R\times S^1$ yields
that the energy (which is the Hamiltonian)
$$
E(u)=\frac{1}{2}\int\big((\partial_x u)^2+(\partial_{x}^{-1}\partial_{y}u)^2\big)-\frac{1}{6}\int u^3
$$
is also formally conserved by the KP-I flow.
By using an anisotropic Gagliardo-Nirenberg inequality (see \eqref{sobolev} below),  
the space $Z^1$ can be identified as the energy space for the KP-I equation, i.e. the space
of functions such that the  two first  conservation laws are finite.

In our previous work, we  have proven the instability of $S_c$ as a solution of the KP-I equation \eqref{KPI} for $c$ large enough.
\begin{theoreme}[\cite{RT1}]\label{zak}
$S_c(t,x)$ is orbitally unstable as a solution of the KP-I equation \eqref{KPI}  provided $c>4/\sqrt{3}$. 
More precisely,  for every $s\geq 0$ there exists $\eta>0$ such that for every $\delta>0$ there exists $u_0^{\delta}\in Z^2\cap H^s$ and a time $T^\delta \approx
|\log\delta|$ such that
$$
\|u_0^{\delta}(x,y)-S_c(0,x)\|_{H^s(\R\times S^1)}+\|u_0^{\delta}(x,y)-S_c(0,x)\|_{Z^1(\R\times S^1)}<\delta
$$
and  the (global) solution of the KP-I equation, defined  by  Theorem  \ref{IK} with data $u_0^\delta$ satisfies
\beq
\label{estinstab}
\inf_{a\in\R}\|u(T^\delta,x-a,y)-S_c(T^\delta,x)\|_{L^2(\R\times S^1)}>\eta.
\eeq
\end{theoreme}
For the proof of Theorem~\ref{zak}, we refer to \cite{RT1}.  Another proof based on the existence of a Lax pair for the  KP equation was obtained previously by Zakharov in \cite{Z2}.
The advantage of the approach of \cite{RT1} is that it is flexible enough to be adapted to more general  and  non-integrable settings.
Namely, in \cite{RT2}, we proved that an analogue of Theorem~\ref{zak} holds if we replace $u\partial_x u$ by $u^p\partial_x u$, $p=2,3$, 
 and  we also  succeeded to extend the approach of \cite{RT1} to prove the transverse instability of capillary-gravity waves in the full water waves system, \cite{RT3}.

 Note that equation \eqref{KPI} is reversible in the sense that if $u(t,x,y)$ is solution of \eqref{KPI}, then  $u(-t, -x, y)$ is also a solution.
Since $S_{c}(-t, -x)= S_{c}(t,x)$, this yields that the instability result of Theorem \ref{zak} also  occurs  for  negative times:  the solution
 of \eqref{KPI} with data $u_{0}^\delta (-x, y)$ verifies  the estimate
 \eqref{estinstab}  at the time $-T^\delta$.

Finally, let us point out  that Theorem \ref{zak} does not yield the instability of  small speeds solitary waves. The stability of these waves is the  first question that
 we want to address in this paper.
Let us now state our first  result.
\begin{theoreme}\label{stability}
$S_c(t,x)$ is orbitally stable as a solution of the KP-I equation \eqref{KPI}, provided  $c<4/\sqrt{3}$. 
More precisely, for every $\varepsilon>0$, there exists $\delta>0$ such that such that if the initial data $u_0$ of the KP-I equation \eqref{KPI} satisfies
 $u_{0} \in Z^2(\mathbb{R} \times S^1)$ and
$$
\|u_0(x,y)-S_c(0,x)\|_{Z^1(\R\times S^1)}<\delta
$$
then the solution of the KP-I equation, defined by Theorem \ref{IK} with data $u_0$ satisfies
$$
\sup_{t\in\R}\inf_{a\in\R}\|u(t,x-a,y)-S_c(t,x)\|_{Z^1(\R\times S^1)}<\varepsilon.
$$
\end{theoreme}
 In the case of the gKP equation i.e.  with nonlinearities $u^p\partial_x u$, $p=2,\,3$ instead of $u\partial_x u$,  there is no available  global well-posedness result, nevertheless the arguments that we use for the proof of Theorem \ref{stability}  yield a conditional stability result of solitary waves of small speeds in the sense that the statement of the above Theorem holds as long as the solution  exists. Indeed the three  main 
 points of the proofs  which are  Lemma \ref{lem1} (KdV orbital stability), Lemma \ref{lem2} and the anisotropic Sobolev inequality \eqref{sobolev}
  still hold for $p=2, \, 3$.

Let us notice that Theorem~\ref{stability}  is against  the usual intuition  that the KdV soliton is unstable under the KP-I flow. 

Finally, let us recall that the stability properties of  $S_{c}$  in the KP-II equation 
$$
\partial_t u+u\partial_x u+\partial_x^3 u+\partial_x^{-1}\partial_y^2u=0,\quad x\in\R,\,\, y\in S^1,
$$
are very different.  The spectral stability of the wave was obtained  in \cite{Pego}  for all speeds $c>0$. More recently,  
the full nonlinear stability i.e.  the analogous of Theorem \ref{stability}  with  the space  $Z^1$ changed   in  $L^2$ was also obtained for all  speeds $c>0$ in  \cite{MT}. Moreover in the context of the KP-II equation a suitable asymptotic stability statement is
proved in \cite{MT} while Theorem~\ref{stability} only provides  orbital stability.  

The second goal of this paper is to describe more precisely  the instability statement of Theorem~\ref{zak} by constructing solutions displaying the instability by their 
behavior at infinity. Let us first observe that by changing the frame
$$
v(t,x-ct,y)=u(t,x,y)
$$
it suffices to study the stability of $Q_c(x)=cQ(\sqrt{c}x)$ under the flow of the equation
$$
\partial_t v-c\partial_x v+v\partial_x v+\partial_x^3 v-\partial_x^{-1}\partial_y^2v=0,\quad x\in\R,\,\, y\in S^1,
$$
i.e. we may reduce the matters to studying stationary solutions, by adding the new term $-c\partial_x v$ to the equation.
 We shall thus consider the gKP-I equation in a moving frame
\begin{equation}\label{gKPI}
\partial_t u-c\partial_x u+
u^p\partial_x u+\partial_x^3 u-\partial_x^{-1}\partial_y^2u=0,\quad p=1,2,3,\quad x\in\R,\,\, y\in S^1.
\end{equation}
It is well-known that \eqref{gKPI} still  has stationary solutions of the from $R_{c,p}(x)=c^{\frac{1}{p}}R_{p}(\sqrt{c}x)$, $c>0$, where $R_{p}$ is a rapidly decreasing function. 
This solution is orbitally stable under the flow of the gKdV equation  (see \cite{Weinstein} for example)
\begin{equation*}
\partial_t u-c\partial_x u+
u^p\partial_x u+\partial_x^3 u=0,\quad p=1,2,3,\quad x\in\R.
\end{equation*}
However,  there exists speeds $c$  such that it is unstable as a solution of \eqref{gKPI}. This fact was shown in \cite{RT2}
 Theorem 1. Here we shall extend and make more precise  the result of \cite{RT2} by showing the following statement.
\begin{theoreme}\label{sofia}
Let $s\geq 0$.  For $p=1, \, 2, \, 3$,  there exists $c_{p}^\star>0$ such that for every $c>c^\star_{p} $ the following holds true.
There exists $u_0\in (H^s\cap Z^1)(\R\times S^1)$ and $u\in C([0,\infty);H^s\cap Z^1)$ a solution of \eqref{gKPI} such that $\partial_y u_0\neq 0$ and
\begin{equation}\label{limit}
\lim_{t\rightarrow\infty}\|u(t,x,y)-R_{c,p}(x)\|_{H^s\cap Z^1}=0.
\end{equation}
\end{theoreme}
In the case of the equation \eqref{KPI}, we have $c_{1}^\star= 4/\sqrt{3}$. For the gKP equation, we shall only prove
 the existence of $c^\star_{p}$ by an abstract argument. Note that the $c^\star_{p}$ that we shall provide is optimal
  in the sense that for $c<c^\star_{p}$, we have  the conditional stability  of the solitary wave (see Remark \ref{optim}).
Let us observe that the above statement implies the instability of $R_{c,p}$, as a solution of \eqref{gKPI} in the sense of Theorem~\ref{zak}.
Indeed let us show that the above statement implies instability in $Z^1$. We reason by contradiction.
Suppose that $R_{c,p}$ is stable in $Z^1$ (see the statement of Theorem~\ref{stability}). 
The problem is invariant by time translations and thus by applying the stability assumption for $t\gg 1$ thanks to 
\eqref{limit}, and using the stability conclusion backwards in time, we get that for every $\varepsilon>0$ there exists $\alpha(\varepsilon)$ such that
$$
\|u_{0}(x,y)-R_{c,p}(x-\alpha(\varepsilon))\|_{Z^1}<\varepsilon.
$$
On the other hand, using a Fourier expansion in $y$, we have that $u_0(x,y)=v_0(x)+w_{0}(x,y)$, $w_0\neq 0$ and $w_0$ is with zero mean in $y$.
Thus
$$
\|u_{0}(x,y)-R_{c,p}(x-\alpha(\varepsilon))\|_{Z^1}
\geq
\|u_{0}(x,y)-R_{c,p}(x-\alpha(\varepsilon))\|_{L^2}
\geq \|w_0(x,y)\|_{L^2}.
$$
Therefore we obtain that $\|w_0\|_{L^2}<\varepsilon$ for every $\varepsilon>0$, i.e. $w_0=0$ which is a contradiction.

 The result of Theorem~\ref{sofia} may be compared to the classical results of existence of strongly stable manifolds for 
 ordinary differential equations. Indeed,   for $c>c^*_{p}$, the linearization of \eqref{gKPI} about $R_{c,p}$ has an unstable  positive eigenvalue,
  hence by the symmetry of the spectrum due to the Hamiltonian structure, there is also a negative eigenvalue and one
   can find the most negative one. To get Theorem \ref{sofia}, we shall prove that  there is a strongly stable manifold to $R_{c,p}$
    associated to this most negative eigenvalue.  Note that Theorem \ref{sofia} provides an example of a global in time
     solution different from the solitary waves  (recall that for $p=2, \, 3$ no large data global existence result is known).
  Results   of this type were  proven by  
Duyckaerts and Merle \cite{DM}  in the setting  of the Nonlinear Schr\"odinger equation and in \cite{Combet}
 for the gKdV equation for $p > 4$ (note that  we have restricted our study to $p=1, \, 2, \, 3$ since for $p \geq 4$, $R_{c,p}$ is already unstable as a solution of the gKdV equation).
 The main difference
with \cite{DM}, \cite{Combet}  is that here we  have  a quasilinear problem for which  the semi-linear fixed point argument used in these works cannot 
 be used.
We use instead a quasilinear scheme based on the simplest  high order energy method (i.e. the classical one for hyperbolic systems).
    
The remaining part of this paper is organized as follows. In the next section, we prove Theorem~\ref{stability}. Then we present the proof of Theorem~\ref{sofia}.
\section{Proof of the stability theorem}
In this section we prove Theorem~\ref{stability}. We need to study the stability of
$Q_c(x)=cQ(\sqrt{c}x)$ as a solution of 
\begin{equation}\label{KPIpak}
\partial_t v-c\partial_x v+
v\partial_x v+\partial_x^3 v-\partial_x^{-1}\partial_y^2v=0,\quad x\in\R,\,\, y\in S^1.
\end{equation}
Consider the energy (Hamiltonian) associated to \eqref{KPIpak}
$$
H(v)\equiv\int_{-\infty}^{\infty}\int_{0}^{2\pi}\Big[(\partial_x v)^2+(\partial_x^{-1}\partial_{y}v)^2+cv^2-\frac{1}{3}v^3\Big]dxdy.
$$
The quantity $H(v)$ is invariant under the flow of \eqref{KPIpak} established in \cite{IK}.
Write the Taylor expansion
$$
H(Q_c+w)=H(Q_c)+DH(Q_c)[w]+\frac{1}{2}D^2H(Q_c)[w,w]+\frac{1}{6}D^3H(Q_c)[w,w,w],
$$
where 
\begin{eqnarray*}
DH(Q_c)[w] & = & \int_{-\infty}^{\infty}\int_{0}^{2\pi}
\Big[2Q_c'\partial_x w+2cQ_cw-Q_c^2w\Big]dxdy
\\
& = &
\int_{-\infty}^{\infty}\int_{0}^{2\pi}
2w\Big[-Q_c''+cQ_c-\frac{1}{2}Q_c^2\Big]dxdy=0,
\end{eqnarray*}
thanks to the equation solved by $Q_c$. Next
$$
D^2H(Q_c)[w,w]=
2\int_{-\infty}^{\infty}\int_{0}^{2\pi}\Big[(\partial_x w)^2+(\partial_x^{-1}\partial_{y}w)^2+cw^2-Q_c w^2\Big]dxdy
\equiv B^c(w,w).
$$
Of course the quadratic form  $B^c(w,w)$ is induced by  the symmetric form
$$
B^c(w_1,w_2)=\frac{1}{4}(B^c(w_1+w_2,w_1+w_2)-B^c(w_1-w_2,w_1-w_2)).
$$
Finally,
$$
D^3H(Q_c)[w,w,w]=-2\int_{-\infty}^{\infty}\int_{0}^{2\pi}w^3 dxdy.
$$
For $f_1,f_2\in H^1(\R)$, such that $\partial_x^{-1}f_1,\partial_x^{-1}f_2\in L^2(\R)$ and $k\in \Z^{\star}$, we set
$$
B^c_k(f_1,f_2)=2\int_{-\infty}^{\infty}
\Big[f_1'f_2'+k^2(\partial_x^{-1}f_1)(\partial_x^{-1}f_2)+ cf_1f_2-Q_cf_1f_2\Big]dx.
$$
Next, for $f_1,f_2\in H^1(\R)$, we set
$$
B^c_0(f_1,f_2)=2\int_{-\infty}^{\infty}
\Big[f_1'f_2'+ cf_1f_2-Q_cf_1f_2\Big]dx.
$$
Note that $B^c_{0}$ is the bilinear form associated to the second derivative of the KdV Hamiltonian about the KdV solitary wave.
We  thus have the following classical stability  property:
\begin{lem}\label{lem1}
There exists $C>0$ such that for every $c>0$, every $g\in H^1(\R)$ satisfying
$$
\int_{-\infty}^{\infty}g(x)Q_c(x)dx=\int_{-\infty}^{\infty}g(x)Q_c'(x)dx=0
$$
one has
$$
B_0^c(g,g)\geq C\Big(\|g'\|_{L^2}^2+c\|g\|_{L^2}^2\Big).
$$
\end{lem}
\begin{proof}
Set $f(x)=g(c^{-\frac{1}{2}}x)$. Then
$$
\int_{-\infty}^{\infty}f(x)Q(x)dx
=c^{-\frac{1}{2}}\int_{-\infty}^{\infty}g(x)Q_c(x)dx=0.
$$
Similarly
$$
\int_{-\infty}^{\infty}f(x)Q'(x)dx=0.
$$
Therefore by the stability theory of $Q$ as solutions of the KdV equation (see e.g. \cite{BSS} and the references therein) we have
$$
B_{0}^{1}(f,f)\geq C\|f\|_{H^1}^2\geq C\|f\|_{L^2}^2\,.
$$
A direct computation shows that
$$
B_{0}^{1}(f,f)=c^{-\frac{1}{2}}B_{0}^{c}(g,g),\quad \|f\|_{L^2}=c^{\frac{1}{4}}\|g\|_{L^2}\,.
$$
Thus 
\begin{equation}\label{ed}
B_{0}^{c}(g,g)\gtrsim c\|g\|_{L^2}^2\,.
\end{equation}
On the other hand, coming back to the definition of $B_{0}^c(g,g)$, we obtain that
\begin{equation}\label{dv}
B_{0}^c(g,g)\geq 2\|g'\|_{L^2}^2-4c\|g\|_{L^2}^2.
\end{equation}
The claim of Lemma~\ref{lem1} is now a direct consequence of \eqref{ed} and \eqref{dv}.
\end{proof}
The next step will be to get a bound from below  for  $B_{k}^c(f,f)$ when $k \neq 0$. This is
 the key point of the proof of Theorem \ref{stability}. 
 We shall first state a robust result  for small enough speeds which can be easily transfered to the gKP-I equation.
\begin{lem}\label{lem2}
There exists $C>0$ such that for every $c<1$, every $k\in \Z^{\star}$ and every $f \in H^1$ such that  $\partial_{x}^{-1}f\in L^2$,
we have the estimate
$$
B_k^c(f,f)\geq
C(1-c)
\Big(\|f\|_{H^1}^2+k^2\|\partial_x^{-1}f\|_{L^2}^2\Big).
$$
\end{lem}
\begin{proof}
Since ${\rm ch}(x)\geq 1$ for every $x\in\R$, we can write
\begin{eqnarray*}
B_k^c(f,f)  & \geq & 2\int_{-\infty}^{\infty}\Big[(f'(x))^2+(\partial_{x}^{-1}f(x))^2+c(f(x))^2-3c(f(x))^2\Big]dx
\\
& \gtrsim &
\int_{-\infty}^{\infty}|\hat{f}(\xi)|^2(|\xi|^2+|\xi|^{-2}-2c)d\xi
\\
& \gtrsim &
(1-c)\|f\|_{L^2}^2\,. 
\end{eqnarray*}
On the other hand coming back to the definition of $B_k^c(f,f)$, we obtain 
\begin{equation*}
B_k^c(f,f)   \geq  2\|f'\|_{L^2}^2+2k^2\|\partial_{x}^{-1}f\|_{L^2}^2-4c\|f\|_{L^2}^2.
\end{equation*}
Thus
$$
\|f'\|_{L^2}^2+k^2\|\partial_{x}^{-1}f\|_{L^2}^2\lesssim 
B_k^c(f,f)+\frac{1}{1-c} B_k^c(f,f)\lesssim  \frac{1}{1-c}B_k^c(f,f). 
$$
This completes the proof of Lemma~\ref{lem2}.
\end{proof}
 In the case of the KP-I equation, we can get a more precise statement:
\begin{lem}\label{lem2bis}
There exists $C>0$ such that for every $c<4/\sqrt{3}$, every $k\in \Z^{\star}$ and every $f \in H^1$ such that $\partial_{x}^{-1}f\in L^2$, we have the estimate
$$
B_k^c(f,f)\geq
C
\Big(\|f\|_{H^1}^2+k^2\|\partial_x^{-1}f\|_{L^2}^2\Big).
$$
\end{lem}
An analogous statement holds for the gKP-I equation  see Remark \ref{optim} below.
\begin{proof}
Again, we can use that 
$$ B_k^c(f,f)   \geq  2\int_{-\infty}^{\infty}\Big[(f'(x))^2+(\partial_{x}^{-1}f(x))^2+c(f(x))^2-Q_c(f(x))^2\Big]dx.$$
Next, let us set $u=\partial_{x}^{-1}f, $ then we have
\beq
\label{Bkb} B_k^c(f,f)   \geq 2 \mathcal{Q}(u, u), \quad   u=\partial_{x}^{-1}f,
\eeq
where 
$$\mathcal{Q}(u, u)= \int_{-\infty}^{\infty}\Big[(u''(x))^2+(u(x))^2+c(u'(x))^2-Q_{c}(u'(x))^2\Big]dx$$
and we note that $\mathcal{Q}$ is the symmetric form associated to the operator
$$ M_{c}u= u^{(4)} -  \big( (c - Q_{c})u'\big)' + u.$$
Note that $M_{c}$ has a self-adjoint realization on $L^2$ with domain $H^4$.
Since $Q_{c}$ tends to zero exponentially fast at $\pm \infty$, we get 
 that $M_{c}$ is a compact perturbation of the operator
  $ \partial_{x}^4 - c \partial_{x}^2 + 1$ and by using the Weyl Lemma, we   get  that  the  essential spectrum
 of $M_{c}$ is   in  $[1, + \infty[$. To get a  useful bound from below for $\mathcal{Q}$, we thus  only need to  prove that $M_{c}$
  has no nonpositive  eigenvalue. Therefore,   let us  study the eigenvalue problem
  $$ M_{c}u = \lambda u, \quad \lambda \leq 0, \quad u \in H^4.$$
  Note that a nontrivial solution $u$ of this problem is actually in $H^\infty$ and that
  one can write $u= v'$ with $v \in H^\infty$. The problem for $v$ reads
  $$ v^{(4)} - (c-Q_{c}) v'' + v= \lambda v.$$
  Motivated by the expression \eqref{soliton} for $Q_{c}$, we set  
  $v(x) = w\big({ x\, \sqrt{c} \over 2}\big)$
  and we get for $w(y)$ the eigenvalue problem
  \beq
  \label{eigw}
  w^{(4)} - 4\big( 1- {3 \over \mbox{ch}^2\, y}\big)w'' + 3 \nu^2 w = 0.
  \eeq
  where $\nu >0$ is defined by 
  \beq
  \label{nudef}
  3\nu^2= {16\over c^2}(1-\lambda).
  \eeq
Consequently, the equation \eqref{eigw} has the same form as the eigenvalue problem studied in \cite{Pego} (see also the Appendix
 of \cite{RT1}). It corresponds to the case $\lambda =0$ in the notation of \cite{Pego}. This case was not studied
  in \cite{Pego}, nevertheless the same kind of arguments can be used to analyze it.
 
 Since $1/\mbox{ch}^2\, y$ tends to zero exponentially fast, the solutions of \eqref{eigw} behave at infinity
  like  the solutions of the constant coefficient equations
 $$ w^{(4)} - 4 w'' + 3 \nu^2w= 0.$$
 The characteristic values of this linear equation are roots of
 \beq
 \label{charac} \mu^4 - 4 \mu^2 + 3 \nu^2= 0.\eeq
 Consequently, we get that for $\nu >0$ there are two roots of positive real part and two roots of negative real parts.
 Moreover, for $\nu^2 \neq 4/3$  the roots are simple.
 Next, by using \cite{Pego}, we get that for  every characteristic value $\mu$, 
 $$ g_{\mu}(y)= e^{\mu \,y} \big( \mu^3 + 2 \mu - 3 \mu^2 \mbox{tanh }y\big)$$
 is an exact solution of \eqref{eigw}. 
   This yields that  for $\nu^2 \neq 4/3$, $(g_{\mu_{1}}, g_{\mu_{2}})$ with $\mu_{1}$, $\mu_{2}$ the two characteristic values of positive real parts
    is a basis of the solutions of \eqref{eigw} which tend to zero when $y$ tends to $-\infty$. Consequently, there is
      a nontrivial  $L^2$ solution of \eqref{eigw} if and only if there exists a nontrivial linear combination of 
       $g_{\mu_{1}}, g_{\mu_{2}}$ which tends to zero when $y $ tends to $+\infty$.
        If such a linear combination exists, since $\mu_{1} \neq \mu_{2}$ this implies that
         $$\lim_{y \rightarrow +\infty} e^{-\mu_{i} \, y}g_{\mu_{i}}(y)= 0, \quad i=1, \, 2.$$
         This yields $\mu_{i}^3 + 2 \mu_{i} -3 \mu_{i}^2=0$, i.e. $\mu_{i}= 0, \, 1,$ or $2$.
         Obviously, for $\nu>0$,   $0$ and $2$ are not solutions of \eqref{charac} while
          for $\mu_{i}= 1$, we obtain
          $\nu ^2= 1$ and hence by using \eqref{nudef} we obtain
          $$ c^2= {16 \over 3}(1-\lambda).$$
          Consequently, there is a nontrival solution with $\lambda \leq 0$ if and only if
           $c^2 \geq {16 \over 3}.$
        As in \cite{Pego},  we get the same result in the case that  $\mu$ is a  double root by using  $(g_{\mu},{ \partial g \over \partial \mu})$
        as a basis of solutions.
        
     We have thus proven that for $c<4/\sqrt{3}$, the eigenvalues of $M_{c}$ below the essential spectrum are positive, 
      this yields that there exists $\alpha >0$ such that
      $$ \mathcal{Q}(u,u)=  (M_{c}u, u) \geq \alpha \|u\|_{L^2}^2, \quad \forall u \in H^4.$$ 
      Consequently, thanks to \eqref{Bkb}, we get that
     $$B_{k}^c(f,f) \geq 2 \alpha \|\partial_{x}^{-1} f \|_{L^2}^2.$$
     On the other hand,  the  bound from below for $B_{k}^c$ used in the proof of Lemma \ref{lem2} yields
     $$ B_{k}^c(f,f) \geq  2 \int_{-\infty}^{+\infty} \Big( |\xi|^2 + {1 \over |\xi|^2} - 2 c \Big) |\hat{f}(\xi)|^2.$$
     Consequenty, by combining the two estimates we get that  for every $A>0$, there exists $C(A)>0$
      such that
      $$ C(A) B_{k}^c(f,f) \geq \int_{-\infty}^{+\infty} \Big( |\xi|^2 + {1 + A \over |\xi|^2} - 2 c \Big) |\hat{f}(\xi)|^2
       \geq \int_{-\infty}^{+\infty}  \Big(  2 \sqrt{1+ A}   -2 c  \Big)
        |\hat{f}(\xi)|^2.$$
        By choosing $A$  such that $\sqrt{1+ A}   >c$ we thus get that
        $$ B_{k}^c(f,f) \gtrsim \|f\|_{L^2}^2$$
        and we end the proof as in the proof of Lemma \ref{lem2}. 
     \end{proof}

Let us now come back to the proof of Theorem~\ref{stability}. We first  use the implicit function theorem to the map
$F:Z^1\times\R\rightarrow \R$, defined by
$$
F(u,\beta)=\int_{-\infty}^{\infty}\int_{0}^{2\pi}u(x+\beta,y)Q_c'(x)dxdy.
$$
Since
$$
F(Q_c,0)=0,\quad \frac{\partial F}{\partial\beta}(Q_c,0)=2\pi\|Q_c'\|_{L^2}^2\neq 0,
$$
we obtain that if the initial data is close to $Q_c$ in $Z^1$ then there exists a modulation parameter $\gamma(t)$, defined at least for small times, so that
$$
v(t,x+\gamma(t),y)=Q_c(x)+w(t,x,y)
$$
with
\begin{equation}\label{OC}
\int_{-\infty}^{\infty}\int_{0}^{2\pi}w(t,x,y)Q_c'(x)dxdy=0.
\end{equation}
Recall the conservation law
$$
H(v(t))=H(v(0))=H(Q_c+w(0)),
$$
where $w(0)$ is small in $Z^1$.
On the other hand
$$
H(v(t))=H(Q_c+w(t))=H(Q_c)+B^c(w(t),w(t))-\frac{1}{2}\int_{-\infty}^{\infty}\int_{0}^{2\pi}w^3(t,x,y) dxdy.
$$
By the anisotropic Sobolev inequality 
\begin{equation}\label{sobolev}
\|u\|_{L^p}\leq C\|u\|_{L^2}^{\frac{6-p}{2p}}\|\partial_x u\|_{L^2}^{\frac{p-2}{p}}
\|\partial_x^{-1}\partial_y u\|_{L^2}^{\frac{p-2}{2p}}\,,\quad 2\leq p\leq 6,
\end{equation}
we infer that
$$
\Big|\int_{-\infty}^{\infty}\int_{0}^{2\pi}w^3(t,x,y) dxdy\Big|\lesssim \|w(t)\|_{Z^1}^3\,.
$$
For the proof of \eqref{sobolev}, we refer to \cite{BIN} or \cite{MST07} (Lemma~2, page~783). 

Next, we can write
$$
B^c(w(t),w(t))=
B^c_{0}(\hat{w}(t,\cdot,0),\hat{w}(t,\cdot,0))+\sum_{k\in\Z^{\star}}B^c_{k}(\hat{w}(t,\cdot,k),\hat{w}(t,\cdot,k)),
$$
where here we use the notation
$$
\hat{w}(t,x,k)=(2\pi)^{-1}\int_{0}^{2\pi}e^{-iky}w(t,x,y)dy,
$$
for  the partial Fourier transform of $w$ with respect to the periodic variable $y$.
Let us observe that
$$
w(t,x,y)=\hat{w}(t,x,0)+\sum_{k\in\Z^{\star}}e^{iky}\hat{w}(t,x,k)
$$
and that  (the  time $t$  being  a passive parameter)
$$
\|w(t)\|_{Z^1}^2\approx
\|\hat{w}(t,\cdot,0)\|_{H^1(\R)}^2+
\sum_{k\in\Z^{\star}}
\Big(\|\hat{w}(t,\cdot,k)\|_{H^1(\R)}^2+k^2\|\partial_{x}^{-1}\hat{w}(t,\cdot,k)\|_{L^2(\R)}^2\Big).
$$

By the orthogonality condition \eqref{OC},  we obtain that 
\begin{equation}\label{OCbis}
\int_{-\infty}^{\infty}\hat{w}(t,x,0)Q_c'(x)dx=0,
\end{equation}
since for $k\neq 0$,
$$
\int_{-\infty}^{\infty}\int_{0}^{2\pi}e^{iky}\hat{w}(t,x,k)Q_c'(x)dxdy=0.
$$
Next, let us  assume  that the initial perturbation is such that $\|v(0,\cdot)\|_{L^2(\R\times S^1)}=\|Q_c\|_{L^2(\R\times S^1)}$, then 
 thanks to the conservation of the $L^2$ norm, we get
\begin{equation}\label{rigid}
\|v(t,\cdot)\|_{L^2(\R\times S^1)}=\|v(0,\cdot)\|_{L^2(\R\times S^1)}=\|Q_c\|_{L^2(\R\times S^1)}\,.
\end{equation}
Write
$$
w(t,x,y)=\alpha Q_c(x)+w_1(t,x,y),
$$
where
$$
\int_{-\infty}^{\infty}\int_{0}^{2\pi} Q_c(x)w_1(t,x,y)dxdy=0.
$$
Recall that $v(t,x+\gamma(t),y)=Q_c(x)+w(t,x,y)$ and using \eqref{rigid}, we get
$$
|(Q_c,w(t,x,y))|\leq \frac{1}{2}\|w(t)\|_{L^2(\R\times S^1)}^2,
$$
where $(\cdot,\cdot)$ denotes the $L^2(\R\times S^1)$ scalar product.
On the other hand
$$
(Q_c,w(t,x,y))=\alpha \|Q_c\|_{L^2(\R\times S^1)}^2\,.
$$
Thus 
\begin{equation}\label{zvezda}
|\alpha|\lesssim \|w(t)\|_{L^2(\R\times S^1)}^2\,.
\end{equation}
We now write
$$
B^c(w(t),w(t))=
B^c_{0}(\alpha Q_c+\widehat{w_1}(t,\cdot,0),\alpha Q_c+\widehat{w_1}(t,\cdot,0))+
\sum_{k\in\Z^{\star}}B^c_{k}(\widehat{w_1}(t,\cdot,k),\widehat{w_1}(t,\cdot,k))\,.
$$
Using  Lemma \ref{lem2bis}, for $c<4/\sqrt{3}$ and $k\in \Z^{\star}$,
$$
B^c_{k}(\widehat{w_1}(t,\cdot,k),\widehat{w_1}(t,\cdot,k))\gtrsim
\|\widehat{w_1}(t,\cdot,k)\|_{H^1(\R)}^2+k^2\|\partial_{x}^{-1}\widehat{w_1}(t,\cdot,k)\|_{L^2(\R)}^2.
$$
Next, we expand
\begin{multline*}
B^c_{0}(\alpha Q_c+\widehat{w_1}(t,\cdot,0),\alpha Q_c+\widehat{w_1}(t,\cdot,0))=
\\
B^c_{0}(\widehat{w_1}(t,\cdot,0),\widehat{w_1}(t,\cdot,0))
+
2\alpha B^c_{0}(Q_c,\widehat{w_1}(t,\cdot,0))
+
\alpha^2 B^c_{0}(Q_c,Q_c).
\end{multline*}
Since thanks to our orthogonality conditions we can write
$$
\int_{-\infty}^{\infty}\widehat{w_1}(t,x,0)Q_c(x)dx=\int_{-\infty}^{\infty}\widehat{w_1}(t,x,0)Q_c'(x)dx=0,
$$
we can apply Lemma~\ref{lem1} and get
$$
B^c_{0}(\widehat{w_1}(t,\cdot,0),\widehat{w_1}(t,\cdot,0))\gtrsim \|\widehat{w_1}(t,\cdot,0)\|_{H^1(\R)}^2\,.
$$
On the other hand thanks to \eqref{zvezda}
$$
\alpha^2 B^c_{0}(Q_c,Q_c)\lesssim \|w(t)\|_{L^2(\R\times S^1)}^4\lesssim \|w(t)\|_{Z^1}^4\,.
$$
Next, we can write
$$
\|\widehat{w_1}(t,\cdot,0)\|_{L^2(\R)}\lesssim 
\Big\|\|w_1(t,x,y)\|_{L^2_y}\Big\|_{L^2_x}=\|w_1(t)\|_{L^2(\R\times S^1)}\leq \|w(t)\|_{L^2(\R\times S^1)}\,.
$$
Thus, by invoking again \eqref{zvezda}, we get
$$
|2\alpha B^c_{0}(Q_c,\widehat{w_1}(t,\cdot,0))|\lesssim \|w(t)\|_{Z^1}^3\,.
$$
In summary, we get the bound
\begin{eqnarray*}
B^c(w(t),w(t)) & \gtrsim & \|w_1(t)\|_{Z^1}^2-C\Big(\|w(t)\|_{Z^1}^3+\|w(t)\|_{Z^1}^4\Big)
\\
& \gtrsim &
\|w(t)\|_{Z^1}^2-C\Big(|\alpha|^2+\|w(t)\|_{Z^1}^3+\|w(t)\|_{Z^1}^4\Big)
\\
& \gtrsim &
\|w(t)\|_{Z^1}^2-C\Big(\|w(t)\|_{Z^1}^3+\|w(t)\|_{Z^1}^4\Big).
\end{eqnarray*}
Therefore, we arrive at the bound
$$
\|w(t)\|_{Z^1}^2\lesssim |H(Q_c+w(0))-H(Q_c)|+\|w(t)\|_{Z^1}^3+\|w(t)\|_{Z^1}^4\,.
$$
Thus by a bootstrap argument, we get the stability statement for initial perturbations such that $\|v(0)\|_{L^2}=\|Q_c\|_{L^2}$.

 To get the stability for general perturbations, we use the classical scaling argument.  We note  thanks to \eqref{soliton} that
 $\|Q_c\|_{L^2} = c^{3\over 4}$. Consequently  if $\|v(0) -Q_{c}\|_{L^2}$ is small for some $c<4/\sqrt{3}$, then we can find
  some $\tilde{c}$ close to $c$ (and thus smaller than $4/\sqrt{3}$) such that   $ \|v(0)\|_{L^2}=\|Q_{\tilde{c}}\|_{L^2}$. By the stability property
   already estabished, we get that
   $$ \sup_{t} \inf_{a} \|u(t,x-a, y) - Q_{\tilde c}(x-\tilde c t) \|_{Z^1}\leq \eps$$ 
    and since
   $$\sup_{t} \inf_{a} \|Q_{c}(x-ct - a ) - Q_{\tilde c}(x-\tilde c t) \|_{Z^1}= \|Q_{c} -  Q_{\tilde{c}} \|_{H^1}$$
   is small when $c$ is close to $\tilde c$ the result follows.
This ends the proof of Theorem~\ref{stability}.
\section{Proof of the instability theorem}
In this section, we prove Theorem~\ref{sofia}.
We first perform a scaling argument which allows to  reduce the matters to a fixed speed with varying period and therefore we will enter 
in the framework of our previous works \cite{RT1,RT2}.

Consider thus the equation
\begin{equation}\label{gKPIbis}
\partial_t u+
u^p\partial_x u+\partial_x^3 u-\partial_x^{-1}\partial_y^2u=0
\end{equation}
with a particular solution $R_c(x-ct)=c^{\frac{1}{p}}R(c^{\frac{1}{2}}(x-ct))$ (we shall omit the index $p$ in $R_{c,p}$ and $R_{p}$
 throughout the proof for  the sake of clarity).

We now observe that for $\lambda>0$, if $u(t,x,y)$ is a solution of \eqref{gKPIbis} then so is
$$
u_{\lambda}(t,x,y)=\lambda^{\frac{2}{p}}u(\lambda^3t,\lambda x,\lambda^2 y)\,.
$$

Suppose now that for some $L>0$ we have a solution of \eqref{gKPIbis} for $x\in\R$, 
$y\in \R\slash (2\pi L\Z)$ such that
\begin{equation}\label{(1)}
\lim_{t\rightarrow\infty}\|u(t,x,y)-R(x-t)\|_{(H^s\cap Z^1)(\R\times \R\slash (2\pi L\Z)}=0
\end{equation}
and
\begin{equation}\label{(2)}
\partial_{y}u(0,x,y)\neq 0.
\end{equation}
Set
$$
\widetilde{u}(t,x,y)\equiv
L^{\frac{1}{p}}u(L^{\frac{3}{2}}t, L^{\frac{1}{2}}x, Ly)\,.
$$
Then $\widetilde{u}(t,x,y)$ is a solution of \eqref{gKPIbis} which is $2\pi$ periodic in $y$. Moreover, a direct computation shows that
for a suitable constant $C(L)$,
\begin{multline*}
\|\tilde{u}(t,x,y)-L^{\frac{1}{p}}R(L^{\frac{1}{2}}(x-Lt))\|_{(H^s_{L}\cap Z^1_{L})(\R_{x}\times \R_{y}\slash (2\pi \Z))}
\\
=C(L)\|u(L^{\frac{3}{2}}t,x,y)-R(x-L^{\frac{3}{2}}t)\|_{(H^s\cap Z^1)(\R_{x}\times \R_{y}\slash (2\pi L \Z))}\,
\end{multline*}
where  $H^s_{L}(\mathbb{R}\times\R_{y}\slash (2\pi  \Z)) $  and $Z^1_{L}(\mathbb{R}\times\R_{y}\slash (2\pi  \Z)) $ are the standard spaces equipped with  the equivalent  norms
$$ \|u\|_{H^s_{L}} = \|  (1+ |\xi/L^{1\over 2}|^s + |k|^s)\hat{u}(\xi,k)\|_{L^2(\mathbb{R}_{\xi}\times \mathbb{Z}_{k})}, \quad
 \|u\|_{Z^s_{L}} = \|  (1+ |\xi/L^{1\over 2}|^s + | (\xi/L^{1\over 2 })^{-1}k|^s)\hat{u}(\xi,k)\|_{L^2(\mathbb{R}_{\xi}\times \mathbb{Z}_{k})}$$
Therefore $\widetilde{u}(t,x,y)$ satisfies the conclusion of Theorem~\ref{sofia}.

 We shall take $L=c$ in the following in order to study the stability of the solitary wave  with speed one submitted
  to $L$ periodic transverse perturbations.
We thus  restrict our attention to the construction of a solution $u$ of \eqref{gKPIbis} satisfying \eqref{(1)} and \eqref{(2)}.

We shall need to construct a dynamics asymptotically close to $R$ under the flow of the equation
\begin{equation}\label{star}
\partial_{t}u +Au+\frac{1}{p+1}\partial_x(u^{p+1})=0,
\end{equation}
posed on $\R_{x}\times \R_{y}\slash (2\pi L \Z)$ for a suitable value of $L$, where $A$ is defined as follows
$$
A\equiv \partial_{x}^3-\partial_x-\partial_{x}^{-1}\partial_{y}^2\,.
$$
Following \cite{Grenier} and our previous works \cite{RT1,RT2}, we look for a solution of \eqref{star} under the form 
\begin{equation}\label{anz}
u(t)=u_{ap}(t)+v(t),\quad t\geq 0,
\end{equation}
with
$$
u_{ap}(t)=\sum_{k=0}^M \delta^k u_{k}(t),
$$
where
 $|\delta | \ll 1$, $M\gg 1$ and $u_k(t)$ are defined iteratively starting from $u_0=R$.
The second term is defined as
\beq
\label{u1}
u_1(t)=e^{-\sigma t}\varphi,
\eeq
where
$\varphi(x,y)=e^{i\frac{n_{0}y}{L}}\psi(x)$  for some $n_{0} \neq 0$ with $\psi\in \cap_{s}H^{s}(\R)$ is such that
\begin{equation}\label{eigen}
{\mathcal A}\varphi=-\sigma\varphi,\quad \sigma>0,
\end{equation}
where ${\mathcal A}w=Aw+\partial_x(R^p w)$ is the linearized operator around $R$
 and $\sigma$ is  the largest  value for which the above equation has a nontrivial solution.
 Note that   if   $u(t,x,y)$ solves
 \beq
 \label{alin}\partial_{t} u + \mathcal{A} u =0
 \eeq
 then $v(t,x,y)= u(-t, -x, -y)$ also solves this equation. This implies that  if $\lambda$ is an eigenvalue of $\mathcal{A}$
  so is $\mathcal{-\lambda}$. Consequently the $\sigma$ we are looking for is exactly the one for which \eqref{alin}
  has an unstable eigenmode  with maximal growth rate under the form $e^{\sigma t } e^{i\frac{n_{0}y}{L}}\psi(x)$.
 In  \cite{RT1}, for $p=1$,  by using \cite{Pego},  we have shown that such a solution exists  for $L>4/\sqrt{3}$. In the case $p=2,3$
the existence of an interval of $L's$ such that the problem corresponding to \eqref{eigen}  has a nontrivial solution, is shown in \cite{RT2}
 \cite{RT4}.
We can actually sharpen this result as follows.
\begin{lem}
\label{lem+}
There exists $L^\star >0$ such that for every $L>L^*$, there exists a solution of  \eqref{eigen} under the form
$\varphi(x,y)=e^{i \frac{y}{L}}\psi(x)$   with $\psi\in \cap_{s}H^{s}(\R)$ and
 $\partial_{x}^{-1} \psi \in L^2(\R)$. 
\end{lem}
\begin{proof}
By using the above remark, it suffices to construct a solution 
of $\mathcal{A} \varphi= \sigma \varphi, \quad \sigma >0$.
 This  eigenvalue problem can be set in the framework of \cite{RT4}. We get that $\psi$ solves 
$$\sigma  \psi= \partial_{x} \big( - \partial_{xx}  - k^2 \partial_{x}^{-2}+  1 -  R^{p} \big) \psi.$$
  We are interested in values of $k$ under the form  $k={1 \over L}$. 
    For $k \neq 0$,  $\psi$ is necessarily under  the form $\psi= \partial_{x} U$ and 
  we get  that $U$ solves
\beq
\label{eqU} - \sigma  \partial_{x} U=\Big(-  \partial_{x}(  - \partial_{xx}  +1 -    R^{p} ) \partial_{x}  +k^2 \Big)U.
\eeq
 Therefore, this eigenvalue problem  can be put in the framework of \cite{RT4} with
 $$ A(k)= - \partial_{x}, \quad L(k)
  =-  \partial_{x}( -  \partial_{xx}  + 1 -   R^{p} ) \partial_{x}  +k^2.$$
  By using \cite{RT4}, we  get that there exists $k_{0}>0$ such that $L(k_{0})$ is nonnegative and  has a one-dimensional kernel.
   Moreover, since $L'(k)$ is positive in the sense of symmetric operators, we get that there exists a unique
    $k_{0} \neq 0$  such that $L(k_{0})$ has a non-trivial kernel.
   Moreover,  thanks to the implicit function Theorem  we have shown in \cite{RT4} that  for every $\sigma$  real and close to zero, there exists  $k(\sigma)$, $U(\sigma)$,
    depending smoothly on $\sigma$, and solutions  
    of \eqref{eqU} such that $k(0)= k_{0}$ and $U(\sigma)= \chi + W(\sigma)$, $W(0)= 0$, with
     $\chi$ an element of the kernel of $L(k_{0})$ and 
    \beq
    \label{contrainte1} (W(\sigma), \chi)=0, \quad \|\chi \|_{L^2(\mathbb{R})}=1
    \eeq
    where  $(\cdot, \cdot)$ stands for the $L^2(\mathbb{R})$ scalar product.
    
    By taking the derivative of \eqref{eqU} with respect to $\sigma$, we first obtain that
  $$ -\partial_{x}\chi=k'(0) L'(k_{0}) \chi + L(k_{0}) W'(0).$$
   Consequently, by taking the scalar product with $\chi$, we get   that
   \beq
   \label{der1}
    k'(0)=0, \quad  L(k_{0}) W'(0)= -\partial_{x} \chi.
    \eeq
    Next, we can compute the second derivative. This yields
    $$ -  2\partial_{x} W'(0)=k''(0) L'(k_{0}) \chi + L(k_{0}) W''(0)$$
     and hence by using \eqref{der1},  we obtain that
    $$ k''(0)=  - 2  {(\partial_{x} W'(0), \chi) \over  (L'(k_{0}) \chi, \chi) }=  -  2 { (L(k_{0}) W'(0), W'(0)) \over  (L'(k_{0}) \chi, \chi) } <0.$$
    Indeed, the numerator is positive  by using that $ L(k_{0})$ is positive on the orthogonal of $\chi$ and that   $W'(0)$ is orthogonal  to $\chi$ thanks to \eqref{contrainte1}.
    This proves that  for $\sigma$ close to zero,  we have
    $$ k(\sigma)= k_{0}- \kappa \sigma^2 + \cdots$$
    with $\kappa >0$ and hence that the instability occurs for $k <k_{0}$.
   Moreover, by using the  Appendix of \cite{Pego-Weinstein}, 
     since $L(k)$ has at most one negative eigenvalue, we know  that there exists at most one solution of \eqref{eqU} with $\sigma$ of positive real part 
     (and thus 
      that $\sigma$ is necessarily  real).
       Consequently we get that there exists  a continuous curve $\sigma(k)$ defined on a maximal interval $(K^*, k_{0})$ such
        that  $\sigma(k)>0$ for every $k$ in this interval. We claim that $K^*=0$. Indeed,  if $K^*>0$ since $\sigma$  remains  bounded
         (see \cite{RT2})  the only possibility is that $ \lim_{k \rightarrow K^*} \sigma(k)=0$. But this implies that $L(K^*)$ has a non-trivial kernel
     which is a contradiction. Consequently, we get that there is a nontrivial solution of \eqref{eqU} with $\sigma >0$
      for every $k \in (0, k_{0})$. To conclude, it suffices to remember that $k$ must be under the form 
      $ k= 1/L$. Hence we get a solution  for $L>1/ k_{0}$.
   \end{proof}
   
  \begin{rem}
  \label{optim}
   Note that by the above argument, we also get that $L(k)$ has no nonpositive  eigenvalue for $0<|k|<k_{0}$ and hence
    that $L(n/c)$, $n \in \mathbb{Z}$ has no  non positive  eigenvalue for $c<1/k_{0}$. This yields that in the case of the gKP-I equation, 
    Lemma \ref{lem2bis} and hence the formal stability of the solitary wave holds for every $c<{1/k_{0}}$.
     \end{rem}
  By using the above Lemma \ref{lem+}, we get that for $L>L^\star$  there exists a nontrivial solution of \eqref{eigen}. 
   Next as in \cite{RT1}, \cite{RT2}, we can choose $n_{0}\neq 0 $ such that $\sigma $ is maximal.
   
 For the end of the proof,   we shall  only consider the case $p=1$, the analysis for other $p$'s being analogous.

For $k\geq 2$, $u_k(t)$ is defined as a solution of the linear problem
$$
\partial_{t}u+{\mathcal A}u+\frac{1}{2}\partial_{x}\Big(\sum_{j=1}^{k-1}u_j u_{k-j}\Big)=0
$$
such that $\lim_{t\rightarrow+\infty}\|u_k(t)\|_{H^s}=0$, i.e.
$$
u_{k}(t)=-\frac{1}{2}\int_{t}^{\infty}e^{-(t- \tau){\mathcal A}}\partial_{x}\Big(\sum_{j=1}^{k-1}u_j(\tau) u_{k-j}(\tau)\Big)d\tau.
$$
To estimate $u_{k}$, we can use that 
thanks to \cite{RT1,RT2}, we have that  for  every $s\geq 2$,
\beq
\label{semigroup}
\|e^{- t{\mathcal A}(n)}\|_{H^{s+ 1}(\R)\rightarrow H^s(\R)}\leq C(n)e^{(\sigma+\eta)|t|},\quad t  \in \mathbb{R},
\eeq
provided $\eta\in (0,\sigma)$, where ${\mathcal A}(n)$ is the Fourier transform of ${\mathcal A}$ with respect to $y$, i.e. 
$$
{\mathcal A}(n)(w)=(\partial_x^3-\partial_x)w+n^2\partial_{x}^{-1}w+\partial_{x}(Qw)\,.
$$
Indeed, in \cite{RT1,RT2}, (see  Theorem 10 of \cite{RT1}), we have proven  that for every $F(t,x)$ such that
 $ \|F(t)\|_{H^{s+1}(\mathbb{R})} \leq  C_{s+1}e^{\gamma t}, $ $ \gamma  > \sigma$,  the solution of
$$ \partial_{t} u + \mathcal{A}(n) u = \partial_{x} F, \quad u_{/t=0}= 0$$
satisfies the estimate
\beq
\label{estRT1} \|u(t) \|_{H^s} \lesssim C_{s+1}e ^{\gamma  t}, \quad \forall t \geq 0.\eeq
Note that   in the statement of  Theorem 10 of \cite{RT1} the assumption $\gamma \geq   2 \sigma$ is made but that the result indeed holds
 as soon as $\gamma$ is  such that $\gamma >\sigma$ (see the resolvent estimate of Theorem 11).
 From this estimate, one can easily obtain the semigroup estimate \eqref{semigroup}. Since  we want to estimate
  the solution of
  $$ \partial_{t} v + \mathcal{A}(n) v= 0, \quad v_{/t=0}= v_{0}, $$
  we can estimate it by using the decomposition 
  $$ v(t)= w(t)+ u(t)$$
  where 
  $w(t)$ solves
  $$ \partial_{t} w + A(n)w= 0, \quad w_{/t=0}= v_{0}, \quad A(n) = \partial_{x}^3 -\partial_{x} + n^2 \partial_{x}^{-1}$$
   and $u$ solves
   $$ \partial_{t} u + \mathcal{A}(n) u= - \partial_{x}\big( Q w\big), \quad  w_{/t=0}= 0.$$
   From the explicit expression in the Fourier side, we immediately get that for every $s\geq 0$, we have
   $$ \|w(t) \|_{H^s} \leq  \|v_{0}\|_{H^s}$$
    and then by using \eqref{estRT1}, we get that
    $$ \|u(t)\|_{H^{s}} \lesssim e^{\gamma t} \|v_{0}\|_{H^{s+1}}$$
    provided $\gamma >\sigma$.
   The estimate \eqref{semigroup} for $t\geq 0$  then follows. The estimate for negative times is again the consequence
    of the symmetry $t\rightarrow -t$, $x\rightarrow -x$ of the equation.
  
Therefore, using that $u_{1}$ and thus  $u_2$ has a compactly supported Fourier transform in $y$, we infer that for $t\geq 0$,
$$
\|u_2(t)\|_{H^s(\R\times \R\slash (2\pi L \Z))}
\leq C\int_{t}^{\infty}e^{(\tau-t)(\sigma+\eta)}e^{-2\sigma\tau}d\tau\leq Ce^{-2\sigma t}\,.
$$
Similarly, one obtains by induction  that 
\beq
\label{uk}
\|u_k(t)\|_{H^s(\R\times \R\slash (2\pi L \Z))}
\leq Ce^{-k\sigma t}
\eeq
and as a consequence 
\begin{equation}\label{decay}
\|\mathcal R(t)\|_{H^s(\R\times \R\slash (2\pi L \Z))}\leq C\delta^{M+1}e^{-(M+1)\sigma t}\,,
\end{equation}
where
$$
\mathcal R\equiv (\partial_t+A)u_{ap}+\frac{1}{2}\partial_{x}(u_{ap}^2)\,.
$$
Coming back to \eqref{anz}, we obtain that $v(t)$ solves the problem
\begin{equation}\label{pbv}
\partial_{t}v+Av+2\partial_{x}(u_{ap}v)+v\partial_{x}v+\mathcal R=0.
\end{equation}
We shall construct a solution of \eqref{pbv} as a limit of a sequence $(v_n)_{n\geq 0}$ defined as follows.
We set $v_0=0$ and for a given $v_n$, we define $v_{n+1}$ as the solution of the linear problem
\begin{equation}\label{iter}
\partial_{t}v_{n+1}+Av_{n+1}+2\partial_{x}(u_{ap}v_{n+1})+v_n\partial_{x}v_{n+1}+\mathcal R=0
\end{equation}
which vanishes at $+\infty$. Namely, if we denote by $S_{v_n}(t,\tau)$ the flow of the linear problem
\begin{equation}\label{cler}
\partial_{t}u+Au+2\partial_{x}(u_{ap}u)+v_n\partial_{x}u=0,
\end{equation}
we define $v_{n+1}$ as
$$
v_{n+1}(t)=-\int_{t}^{\infty}S_{v_n}(\tau,t)(\mathcal R(\tau))d\tau.
$$
Taking for $s>2$ the $H^s$ scalar product of \eqref{cler} with $u$ gives that the solutions of \eqref{cler} satisfy the energy estimate
$$
\frac{d}{dt}\|u(t)\|_{H^s}^2\leq C\Big(\|u_{ap}(t)\|_{W^{s+1,\infty}}+\|v_n(t)\|_{H^s}\Big)\|u(t)\|_{H^s}^2\,.
$$
Therefore, we obtain that if $v_n$ ranges in a fixed ball of $L^{\infty}([0,+\infty);H^s)$ then
\begin{equation}\label{semi}
\|S_{v_n}(t,\tau)\|_{H^s\rightarrow H^s}\leq e^{\Lambda t},\quad t\geq 0,
\end{equation}
where $\Lambda\geq C+C\delta M$. We now fix $M$ large enough and $\delta$ small enough so that 
$C+C\delta M<\sigma(M+1)$ and choose $\Lambda$ in between. Therefore, using \eqref{decay} and \eqref{semi}, 
we obtain that  for every $s>2$, there exist $C_{s}>0$ and $\alpha>0$ such that for every $n\geq 0$ and
every $t\geq 0$,
\beq
\label{weak}
\|v_{n+1}(t)\|_{H^s}\leq C \delta^{M+1}e^{-(M+1) \sigma  t}.
\eeq

Next since $\partial_{t} v_{n}$ is bounded in $\mathcal{C}([0, T], L^2)$ for every $T$, we get by standard arguments
 that    there exists   $v \in \mathcal{C}(\mathbb{R}_{+}, H^\sigma)$    for every $\sigma <s$ such that,   up to a subsequence, 
  $v_{n}$ converges to $v$ in $\mathcal{C}_{loc}(\mathbb{R}_{+}, H^\sigma)$. This yields that $v$ satisfies the equation \eqref{pbv}.
  Moreover, by passing to the weak limit in \eqref{weak}, we get that $v$ satisfies 
  \begin{equation}\label{borne}
\|v(t)\|_{H^s}\leq Ce^{-\sigma(M+1)  t},\quad t\geq 0.
\end{equation}

Since $u= u_{ap}+ v$, we immedately  get thanks to  \eqref{u1}, \eqref{uk}
 that 
 \beq
 \label{borneHs} \|u - Q \|_{H^{s}(\mathbb{R}\times \mathbb{R}/\slash 2\pi L \mathbb{Z})} \lesssim e^{- \sigma t}.\eeq
Let us now get the claimed bound in $Z^1$. For that purpose it suffices to recall that $u=u_{ap}+v$  solves \eqref{star} and observe that $u$ is a perfect
$x$ derivative, i.e.
$$
u(t)=-\frac{1}{2}\partial_{x}\int_{t}^{\infty}e^{-(t- \tau)A}(u^2(\tau))d\tau.
$$
Indeed, this yields
 $$ \|\partial_{x}^{-1} \partial_{y} u \|_{L^2} \lesssim \int_{t}^{\infty} \|\partial_{y}(u^2)\|_{L^2} d\tau$$
 and since $Q$ does not depend on $y$, we get from \eqref{borneHs} that
 $$  \|\partial_{x}^{-1} \partial_{y} u \|_{L^2} \lesssim e^{-\sigma t}.$$
 
Finally, the condition $\partial_{y}u(0,x,y)\neq 0$ may be achieved for $\delta\ll 1$, since $\partial_{y}u_1(0,x,y)\neq 0$.

This completes the proof of Theorem~\ref{sofia}.
\section{Acknowledgment} 
We benefited from a discussion with Nicolas Burq. The second author is supported by an ERC grant. 

\end{document}